\documentclass[12pt]{article}
\usepackage{amsmath}
\usepackage{amssymb}
\usepackage{amsthm}
\usepackage{graphicx}
\usepackage{pstricks}
\usepackage{lscape}

\theoremstyle{plain}
\newtheorem{theorem}{Theorem}[section]
\newtheorem{proposition}[theorem]{Proposition}
\newtheorem{lemma}[theorem]{Lemma}
\newtheorem{definition}[theorem]{Definition}
\newtheorem{corollary}[theorem]{Corollary}
\newtheorem{conjecture}[theorem]{Conjecture}
\theoremstyle{definition}
\newtheorem{example}[theorem]{Example}
\newtheorem{question}[theorem]{Question}

\newcommand{\GA}{\operatorname{GA}}

\newcommand{\F}{\mathbb{F}}

\newcommand{\lp}{\longrightarrow}
\newcommand{\mb}{\mathbb}

\newcommand{\N}{\mb{N}}

\newcommand{\Aff}{\operatorname{Aff}}
\newcommand{\TA}{\operatorname{TA}}
\newcommand{\kar}{\operatorname{char}}
\newcommand{\ME}{\operatorname{ME}}
\newcommand{\MA}{\operatorname{ME}}
\newcommand{\GL}{\operatorname{GL}}
\newcommand{\Jac}{\operatorname{Jac}}
\newcommand{\m}{\operatorname{m}}
\newcommand{\J}{\operatorname{J}}

\title{Polynomial endomorphisms over finite fields: experimental results}
\author{\begin{tabular}{ll}
Stefan Maubach & Roel Willems\footnote{Funded by Phd-grant of council for the
physical sciences, Netherlands Organisation for scientific
research (NWO).}\\
&\\
\small
Jacobs University &\small Radboud University Nijmegen\\
\small 28759 Bremen &\small Postbus 9010, 6500 GL Nijmegen  \\
\small Germany&\small The Netherlands\\
\small s.maubach@math.ru.nl& \small r.willems@math.ru.nl\\
\end{tabular}}

\begin{document}

\maketitle

\begin{abstract}{Given a finite field $\F_q$ and $n\in \N^*$, one could try to compute all polynomial endomorphisms $\F_q^n\lp \F_q^n$ up to a certain degree with a specific property. We consider the case $n=3$. If the degree is low (like 2,3, or 4) and the finite field is small ($q\leq 7$) then some of the computations are still feasible.
In this article we study the following properties of endomorphisms: being a bijection of $\F_q^n\lp \F_q^n$, being a polynomial automorphism, being a {\em Mock automorphism}, and being a locally finite polynomial automorphism. In the resulting tables, we point out a few interesting objects, and pose some interesting conjectures which surfaced through our computations.}\end{abstract}

\section{Introduction}

{\bf Notations and definitions:} Throughout this paper, $\F_q$ will be a finite field of characteristic $p$ where $q=p^r$ for some $r\in \N^*$.  
When $F_1,\ldots,F_n\in k[X_1,\ldots,X_n]$ ($k$ a field), then $F:=(F_1,\ldots,F_n)$ is a polynomial endomorphism over $k$. 
If there exists a polynomial endomorphism $G$ such that $F(G)=G(F)=(X_1,\ldots,X_n)$, then $F$ is a polynomial automorphism (which is stronger than stating that $F$ induces a bijection on $k^n$).  We will write $X$ for $X_1,\ldots,X_n$. The
polynomial automorphisms in $n$ variables over $k$ form a group, denoted $\GA_n(k)$ (compare the notation $\GL_n(k)$), while the set of polynomial endomorphisms is denoted by $\ME_n(k)$ (the monoid of endomorphisms). 
If $F\in \GA_n(k)$ such that $\deg(F_i)=1$ for all $1\leq i \leq n$, then $F$ is called affine. The affine automorphisms form a subgroup of $\GA_n(k)$ denoted by $\Aff_n(k)$. In case $F\in \GA_n(k)$ such that $F_i\in k[X_i,\ldots,X_n]$ for each $1\leq i\leq n$, then $F$ is called triangular, or Jonqui\'ere. The triangular automorphisms form a subgroup of $\GA_n(k)$, denoted by $\textup{J}_n(k)$. 
The subgroup of $\GA_n(k)$ generated by $\Aff_n(k)$ and $\textup{J}_n(k)$ is called the {\bf tame automorphism group}, denoted by $\TA_n(k)$.  By $\deg(F)$ we will denote the maximum of $\deg(F_i)$. The set of polynomial maps of degree $d$ or less we denote by $\ME^d_n(k)$, and the endomorphisms whose affine part is the identity we denote by
$\overline{\ME}_n(k)$. The following notations are now natural:  $\overline{\ME}^d_n(k):=\overline{\ME}_n(k)\cap \ME_n^d(k)$, $\GA_n^d(k):=\ME_n^d(k)\cap \GA_n(k)$,   $\overline{\GA}_n(k):=\overline{\ME}_n(k)\cap \GA_n(k)$, and 
$\overline{\GA}_n^d(k):=\overline{\GA}_n(k)\cap \GA_n^d(k)$.
If $F,G\in \ME_n(k)$, then $F$ and $G$ are called {\bf equivalent (tamely equivalent)} if there exist $N,M\in \GA_n(k)$ ($N,M\in \TA_n(k)$ such that $NFM=G$. If $F\in \ME_n(k)$ then we say that $(F,X_{n+1},\ldots,X_{n+m})\in \MA_{n+m}(k)$ is a stabilisation of $F$. We hence introduce the terms {\bf``stably equivalent'' and ``stably tamely equivalent''} meaning that a stabilisation of $F$ and $G$ are equivalent or tamely equivalent.\\

The automorphism group $\GA_n(k)$ is one of the basic objects in (affine) algebraic geometry, and the understanding of its structure a highly-sought after question. If $n=1$ then $\GA_1(k)=\Aff_1(k)$, and if $n=2$ then one has the Jung-van der Kulk theorem \cite{Jung, Kulk}, stating among others that $\GA_2(k)=\TA_2(k)$. However, in dimension 3 the structure of $\GA_3(k)$ is completely dark. 
The only strong result is in fact a {\em negative} result by Umirbaev and Shestakov \cite{SU1,SU2}, stating that if $\kar{k}=0$, then 
$\TA_3(k)\not = \GA_3(k)$. 

It might be that all types of automorphisms known in $\GA_3(k)$ have already surfaced, but the possibility exists that there are some 
strange automorphisms that have eluded common knowledge so far. But, in the case $k=\F_q$, we have an opportunity:  one could  simply check the finite set of  endomorphisms up to a certain degree $d$, i.e. $\MA_n^d(\F_q)$, and determine which ones are automorphisms. Any ``new'' type of automorphisms {\em have} to surface in this way.

Unfortunately, the computations rapidly become unfeasible if the degree $d$, the number of variables $n$, or the size of the finite field $\F_q$, are too large. We didn't find any significant shortcuts except the ones mentioned in section \ref{S2}. In the end, for us scanning through lists of $2^{30}=8^{10}$
endomorphisms was feasible, but $3^{20}=9^{10}$ was barely out of reach. 

 In the case of $k=\F_q$, another interesting class that surfaces are the (what we define as) {\bf mock automorphisms} of $\F_q$:
 endomorphisms which induce bijections of $\F_q^n$, and whose determinant of the Jacobian is a nonzero constant. Such maps are interesting for example for cryptography (being ``multivariate permutation polynomials''). 

In this article, we do the (for us) feasible computations, and analyze the resulting lists. In particular, we compute (some of) the (mock) automorphisms for $n=3$, $d\leq 3$, and $q\leq 5$.

\section{Generalities on polynomial automorphisms}
\label{S2}

The following lemma explains why we only study polynomial maps having affine part identity: 

\begin{lemma}
Let $F\in \GA_n^d(k)$. Then there exists a unique $\alpha, \beta \in \Aff_n(k)$ and $ F', F'' \in \overline{\GA}^d_n(k)$ such that  
\[ F=\alpha F' = F''\beta. \]
\end{lemma}
\begin{proof}
For the first equality, take $\alpha$ to be the affine part of $F$, and define $F':=\alpha^{-1}F$.  
For the second, do the first equality for $F^{-1}$, i.e. $F^{-1}=\gamma G$ for some $\gamma\in \Aff_n(k), G\in \overline{\GA}_n(k)$. 
Then $F=G^{-1}\gamma^{-1}$ i.e. take $\beta=\gamma^{-1}$, $F'':=G^{-1}$. The fact that $F''\in \GA_n^d(k)$ is easy to check by comparing the highest degrees of $F''$ and $F$. 
\end{proof}

A useful criterion is that if $F$ is invertible, then $\det(Jac(F))\in k^*$. The converse is a notorious problem in characteristic zero:\\

\noindent
{\bf Jacobian Conjecture:} (Short JC) If $\kar(k)=0$, $F\in \ME_n(k)$, and $\det(\Jac(F))\in k^*$, then $F$ is an automorphism.\\

The JC in $\kar(k)=p$ is not true in general, as already in one variable, $F(X_1):=X_1-X_1^p$ has Jacobian $1-pX_1^{p-1}=1$, but 
$F(0)=F(1)$ and so $F$ is not a bijection. However, the following two (well-known) lemma's show that the Jacobian conjecture is true for the special case where $degree(F)=2$ and $
char(k)\geq 3$.

\begin{lemma}
 Let $F:k^n\rightarrow k^n$ be a polynomial endomorphism of degree $2$ with $det(Jac(F))$ nowhere zero.
Assume that $\kar(k)=p\not= 2$, then $F$ is injective. In particular, if $k$ is a finite field, then $F$ is bijective.
\end{lemma}

\begin{proof}
Suppose $F$ is not injective, then there exist $a,b\in k^n$ such that $F(a)=F(b)$. We may assume that $a=0=F(a)$, as we may replace $F$ by $F(a-X)-F(a)$ if necessary. Now consider $F(tX)=F_0+tF_1(X)+t^2F_2(X)$, where $F_i(X)$ is the homogeneous part of $F(X)$ of degree $i$ and $t$ is a new variable.
Then on the one hand $\frac{d}{dt}F(tX) = F_1(X)+2tF_2(X)$, on the other hand $\frac{d}{dt}F(tX)=Jac(F)\arrowvert_{tX}X$.
Now if we substitute $t=\frac{1}{2}$, then on the one hand we get $\frac{d}{dt}F(tX)\arrowvert_{t=\frac{1}{2}}= F_1(X)+F_2(X)=F(X)$, (by assumption $F(0)=0$, so $F_0=0$). And on the other hand we get $\frac{d}{dt}F(tX)\arrowvert_{t=\frac{1}{2}}=Jac(F)\arrowvert_{\frac{1}{2}X}X$.
But this means that $0=F(a)=Jac(F)\arrowvert_{\frac{1}{2}a}a$ but since $det(Jac(F)\arrowvert_{\frac{1}{2}a})\not=0$ it follows that $Jac(F)\arrowvert_{\frac{1}{2}a}$ is invertible, and this implies that $a=0$, which contradicts our assumption that $a\not=0$.
So $F$ is injective. If $k$ is a finite field then $k^n$ is a finite set, so injective implies bijective.
\end{proof}

\begin{corollary}\label{Jac1aut}
  Let $F:k^n\rightarrow k^n$ be a polynomial endomorphism of degree $2$ with $det(Jac(F))=1$.
Assume that $char(k)\not= 2$, then $F$ is an automorphism.
\end{corollary}
\begin{proof}
 Let $K$ be the algebraic closure of $k$. And consider $F$ as a polynomial endomorphism of $K^n$. For every finite extension $L$ of $k$, we have $F:L^n\rightarrow L^n$ is a bijection, by the above lemma. Hence, $F:K^n\rightarrow K^n$ is a bijection (as $K$ is the infinite union of all finite extensions of $k$). But $K$ is algebraically closed so a bijection of $K^n$ is a polynomial automorphism, so it has an inverse $F^{-1}$. Now Lemma 1.1.8 in \cite{EFC} states that $F^{-1}$ has coefficients in $k$, which means that $F^{-1}$ is defined over $k$, which means that $F$ is a polynomial automorphism over $k$.
\end{proof}

One remark on the previous result about the difference between $det(Jac(F))=1$ and $det(Jac(F))$ is nowhere zero.
If $det(Jac(F))$ is nowhere zero over $k$ this does not imply that $det(Jac(F))$ over $K$ is nowhere zero, consider the following example (see the warning after Corollary 1.1.35 in \cite{EFC}):

\begin{example}
Let $F=(x,y+axz,z+bxy)\in k[x,y,z]$ with $k$ a finite field of characteristic $p$ and $a,b\not= 0\ \in k$, such that $ab$ is not a square. Then $det(Jac(F))=1-abx^2$ is nowhere zero, but obviously $F$ not invertible.
\end{example}

The following result is on subsets of groups that are invariant under a subgroup.

\begin{lemma}\label{FS}
 Let $G$ be a group, $H$ a finite subgroup of $G$ and $V$ a finite subset of $G$ such that $HV\subseteq V$, then $\#H|\#V$.
\end{lemma}
\begin{proof}
Since $G$ acts transitively on $G$, $H$ acts transitively on $V$. Thus, for every $v\in V$, $Hv$ is an orbit set-isomorphic to $H$. 
Also,  $HV=V$ consists of disjoint orbits of the form $Hv$, so $\# H | \# V$.
\end{proof}

In this article we also consider so-called {\em locally finite polynomial automorphisms}. A motivation for studying these automorphisms is that they might generate the automorphism group in a natural way (see \cite{Fu-Mau} for a more elaborate motivation of studying these maps). The reason that we make computations and classifications on them in this article is to have some examples on hand to work with in the future, as there can be rather complicated locally finite polynomial automorphisms. 

\begin{definition} Let $F\in \MA_n(k)$. Then $F$ is called locally finite (short LFPE) if $\deg(F^n)$ is bounded, or equivalently, there exists $n\in \N$ and $a_i\in k$ such that $F^n+a_{n-1}F^{n-1}+\ldots + a_1F+a_0 I=0$. We say that $T^n+a_{n-1}T^{n-1}+\ldots + a_1T+a_0$ is a vanishing polynomial for $F$. In \cite{Fu-Mau} theorem 1.1 it is proven that these vanishing polynomials form an ideal of $k[T]$, and that there exists a minimum polynomial $\m_F(T)$.
\end{definition}

When trying to classify LFPEs and their minimal polynomials (i.e using computer calculations) one can use the following lemmas to reduce computations:

\begin{lemma}\label{LFPEconjugacy}
Let $F\in \GA_n(k)$  and $L \in GL_n(k)$ then $F$ is locally finite iff $L^{-1}FL$ is locally finite. Furthermore if $\m(T)\in k[T]$ is the minimum polynomial for $F$, then $\m(T)$ is also the minimum polynomial of $L^{-1}FL$.
\end{lemma}
\begin{proof}
Suppose $F$ is locally finite with minimum polynomial $\m(T)=m_0+m_1T+\cdots +m_dT^d$, so $m_0I+m_1F+m_2F^2+\cdots + m_dF^d=0$, where $I$ is the identity mam, and $F^d$ is the commosition of $d$ $F$'s.
Now $0=L^{-1}(m_0I+m_1F+m_2F^2+\cdots + m_dF^d)L=m_0L^{-1}IL+m_1L^{-1}FL+m_2L^{-1}F^2L+\cdots + m_dL^{-1}F^dL=
m_0I+m_1L^{-1}FL+m_2(L^{-1}FL)^2+\cdots + m_d(L^{-1}FL)^d=\m(L^{-1}FL)$.
This shows that if $F$ is locally finite with minimum polynomial $\m(T)$, then so is $L^{-1}FL$.
\end{proof}

When classifying LFPEs, one cannot simply restrict to $\overline{\GA}_n(k)$, as it is very well possible that $F\in \overline{\GA}_n(k)$
is not an LFPE, but $\alpha F$ is where $\alpha\in \Aff_n(k)$. However, we can restrict to classes of affine parts under linear maps, by the following lemma:

\begin{lemma}\label{CompLocFin}
Let $\alpha\in Aff_n(k)$ and $F\in \overline{GA}^d_n(k)$. Suppose that $\beta=L^{-1}\alpha L$ , where $L\in GL_n(k)$, and that $\beta F$ is locally finite. Now there exists an automorphism $G\in \overline{GA}^d_n(k)$, such that $\beta F=L^{-1}\alpha G L$, i.e. $\beta F$ is in the conjugacy class of $\alpha G$. Furthermore, the minimum polynomials of $F$ and $G$ are the same. 
\end{lemma}
\begin{proof}
Just take $G=LFL^{-1}$, then $L^{-1}\alpha GL=L^{-1}\alpha LFL^{-1}L=L^{-1}\alpha LF=\beta F$.
\end{proof}

So in order to classify the locally finite automorphisms (up to some degree $d$), it suffices to compute the conjugacy classes of $\Aff_n(k)$ under conjugacy by $\GL_n(k)$, and compose a representative of each class with all the elements of $\GA_n^d(k)$.

When considering LFPEs over finite fields, we have the additional following lemma:

\begin{lemma} \label{order} Let $F\in \GA_n(\F_q)$ be an LFPE. Then $F$ has finite order (as element of $\GA_n(\F_q)$).
\end{lemma}

\begin{proof} If $F$ is an LFPE, then there exists a minimum polynomial $\m(T)$ generating the ideal of vanishing polynomials for $F$. 
There exists some $r\in \N$ such that $\m(T)~|~ T^{q^r}-T$, yielding the result.
\end{proof}

Another concept that surfaces, is the following:

\begin{definition} Let $F=I+H \in \overline{\MA}_n(k)$ where $H=(H_1,\ldots, H_n)$ is the non-linear part. Then $F$ is said to satisfy the {\em dependence criterion} if $(H_1,\ldots, H_n)$ are linearly dependent.
\end{definition}

Notice that $F\in \overline{\MA}_n(k)$ satisfying the dependence criterion is equivalent to being able to apply a linear conjugation to isolate one variable, i.e. $L^{-1}FL=(X_1,X_2+H_2, \ldots, X_n+H_n)$ for some linear map $L$.

\section{Computations on endomorphisms of low degree}
It is clear that  $\#\Aff_n(k)|\#\GA_n^d(k)$.
Let $F:k^n\rightarrow k^n$ be an automorphism then we can consider $\alpha$ to be the affine part of $F$, which is obviously invertible and we can then look at $G=\alpha^{-1}F$, where $G$ now has affine part the identity. This means that to compute all automorphisms it suffices to compute all automorphisms having affine part the identity and compose each of them from the left with all affine automorphisms.
This suffices for our computations over $\F_2$ and $\F_3$, but for larger finite fields ($\F_4,\F_5$ and $\F_7$) we will add additionally the dependence criterion.

Finally recall that over $\F_q$ there are $(q^n-1)(q^n-q)\cdots (q^n-q^{n-1})$ linear automorphisms, and that there are $q^n$ times as many affine automorphisms as linear.

For the rest of the article, as we stay in 3 dimensions, we will rename our variables $x,y,z$. 

\subsection{The finite field of two elements: $\mathbb{F}_2$}\label{F2}

As mentioned in the previous section to find all polynomial automorphisms it suffices to find all automorphisms having affine part equal to the identity, and  there are $(2^3-1)(2^3-2)(2^3-2^2)=168$ linear automorphisms. There are $2^3*168=1344$ affine automorphisms.

\subsubsection{Degree 2 over $\F_2$}

We remind the reader that by a {\em mock automorphism} we mean an endomorphism $E\in \MA_n(k)$ such that $E$ induces a permutation of $k^n$ and $\det(\Jac(E))\in k^*$.
Over $\F_2$, Corollary \ref{Jac1aut} does not hold so there do exist mock automorphisms which are not automorphisms in  $\MA_3^2(\F_2)$. 

\begin{theorem}\label{T1} If $F\in \MA^2_3(\F_2)$ is a mock automorphism, then  $F$ is in one of the following four classes:
\begin{itemize}
 \item[1)] The 176 tame automorphisms, equivalent to $(x,y,z)$.
 \item[2)] 48 endomorphisms tamely equivalent to $(x^4+x^2+x,y,z)$.
 \item[3)] 56 endomorphisms tamely equivalent to $(x^8+x^2+x,y,z)$.
 \item[4)] 56 endomorphisms tamely equivalent to $(x^8+x^4+x,y,z)$.
\end{itemize}
In particular, all automorphisms of this type are tame, i.e. $\GA_3^2(\F_2)=\TA_3^2(\F_2)$
Furthermore, the equivalence classes are all distinct, except possibly class (3) and (4) (see conjecture \ref{Qu1}). \\
There are in total $1344\cdot 176=236544$ automorphisms of $\F_2^3$ of degree less or equal to $2$. 
\end{theorem} 

\begin{proof} The classification is done by computer, see \cite{RoelThesis} chapter 5.
We can show how for example $(x^8+x^4+x, y,z)$ is tamely equivalent to a polynomial endomorphism of degree 2:
\[ \begin{array}{l}
(x+y^2,y+z^2,z)(x^8+x^4+x,y,z)(x,y+x^4+x^2,z+x^2)=\\
(x+y^2,y+x^2+z^2,z+x^2)
\end{array} \]
 What is left is to show that the classes (1),(2),  and (3)+(4) are different. 
Class (1) consists of automorphisms while (2),(3),(4) are not.
Using the below lemma \ref{L1}, the endomorphisms of type (2) are all bijections of $\F_{2^m}^3$ if $3\not|m$, and the endomorphisms of type (3) and (4) are all bijections of $\F_{2^m}^3$ if $7\not|m$. 
The last sentence follows since $\#\Aff_3(\F_2)=1344$.
\end{proof}

\begin{lemma}\label{L1}
$x^4+x^2+x$ is a bijection of $\F_{2^r}$ if $3\not |r$, and $x^8+x^4+x$ and $x^8+x^2+x$ are bijection of $\F_{2^r}$ if $7\not | r$.
\end{lemma}

\begin{proof}
Let us do $f(x):=x^8+x^4+x$, the other proofs go similarly. $f$ is a bijection if and only if $f$ is injective if and only if $f(x)=f(y)$ has 
only $x=y$ as solutions. $f(x)=f(y)$ if and only if $(x-y)^8+(x-y)^4+(x-y)=0$. $x=y$ is a solution, another solution would be equivalent to finding a zero of $x^7+x^3+1$. Now it is an elementary exercise to see that if $\alpha\in \F_{2^r}$ is a zero of this polynomial, then
$7| r$. 
\end{proof}

\begin{question}\label{Qu1}
(1)  Are $F=(x^8+x^2+x,y,z)$ and $G=(x^8+x^4+x,y,z)$ (tamely) equivalent? \\
(2) More general: Are $x^8+x^2+x$ and $x^8+x^4+x$ stably (tamely) equivalent?\\
\end{question}

The above question is particular to characteristic $p$, for consider the following:

\begin{lemma} \label{equiv} Let  $P,Q\in k[x]$. Assume that  $F:=(P(x), y, z)$ is equivalent to $G:=(Q(x),y,z)$. 
Then $P'$ and $Q'$ are equivalent, in particular $Q'(ax+b)=cP'$ for some $a,b,c\in k$, $ac\not =0$.
\end{lemma}

\begin{proof} Equivalent means there exist $S, T\in \GA_3(k)$ such that $SF=GT$. Write $\J$ for $\det(\Jac)$. 
Now $\J(S)=\lambda, \J(T)=\mu$ for some $\lambda,\mu\in k^*$. Using the chain rule we have
\[      \begin{array}{ll}
   &\J(SF)=\J(F)\cdot  (\J(S)\circ (F))= \frac{\partial P}{\partial x} \cdot   (\lambda\circ (F)) =\lambda  \frac{\partial P}{\partial x}\\
=&\J(GT)=\J(T)\cdot (\J(G)\circ (T)) =\mu \cdot ( \frac{\partial Q}{\partial x} \circ T)
\end{array}
 \]
so
\[ 
Q'(T)= \frac{\lambda}{\mu}P'\]
which means that $T=(T_1,T_2,T_3)$ and $T_1=ax+b$ where $a\in k^*, b\in k$, proving the lemma.
\end{proof}

\begin{corollary} Assume $\kar(k)=0$. Let  $P,Q\in k[x]$. Assume that  $F:=(P(x), y, z)$ is equivalent to $G:=(Q(x),y,z)$. 
Then $P$ and $Q$ are equivalent.
\end{corollary}

\begin{proof} Lemma \ref{equiv} shows that  $ P' (ax+b)= c Q'$ for some $a,b,c\in k$, $ac\not =0$. In
characteristic zero we can now integrate both sides and get $a^{-1}P(ax+b)=cQ$ proving the corollary.
\end{proof}

Note that in the ``integrate both sides'' part the characteristic zero is used, as $(x+x^2+x^8)'=(x+x^4+x^8)'$ in characteristic $2$. 

Note that all the above one-variable polynomials $x^8+x^2+x, x^4+x^2+x$ have a stabilisation which is tamely equivalent to a polynomial endomorphism of degree 2. In this respect, note the following proposition, which is lemma 6.2.5 from \cite{EFC}:

\begin{proposition}
Let $F\in \ME_n(k)$ where $k$ is a field. Then there exists $m\in \N$, and $G,H\in \TA_{n+m}(k)$ such that (writing the stabilisation of $F$ as $\tilde{F}\in \ME_{n+m}(k)$)  $G\tilde{F}H$ is of degree 3 or less. 
\end{proposition}

\subsubsection{Locally finite in degree 2 over $\F_2$} 

We now want to classify the locally finite automorphisms among the 236544 automorphisms over $\F_2$ of degree 2 (or less), 
and we want to determine the minimum polynomial of each. Using  lemma \ref{LFPEconjugacy}, we may classify up to conjugation by a linear map. 
 We found   $262$ locally finite classes under linear conjugation, with the following minimum polynomials:\\
\ \\
\begin{tabular}{|l|c|c|}
 \hline
Minimumpolynomial & $\#$ & $t$ \\ \hline \hline
$F^5+F^4+F+I$ & $16$ & $8$\\ \hline
$F^4+F^3+F^2+I$ & $8$ & $7$\\ \hline
$F^4+F^3+F+I$ & $26$ & $6$\\ \hline
$F^4+I$ & $12$ & $4$\\ \hline
$F^4+F^2+F+I$ & $8$ & $7$\\ \hline
$F^3+F^2+F+I$ & $139$ & $4$\\ \hline
$F^3+F^2+I$ & $2$ & $7$\\ \hline
$F^3+F+I$ & $2$ & $7$\\ \hline
$F^3+I$ & $14$ & $3$\\ \hline
$F^2+I$ & $34$ & $2$\\ \hline
$F+I$ & $1$ & $1$\\ \hline
\end{tabular}\\
\ \\
In the above tabular, $\#$  denotes the number of {\em conjugacy classes} (i.e. not elements) having this minimum polynomial, while $t$ denotes the order of the automorphism (see lemma \ref{order}).  Furthermore, observe that $\#$ displayed is the number of conjugacy classes that satisfy this relation, not the total number of automorphisms.

\subsubsection{Degree 3 over $\F_2$}

We only comsidered the endomorphisms of the form $F=I+H$, where $H$ is homogeneous of degree $3$. (The automorphisms of degree 3 or less  in general was just out of reach.)
The below tabular describes the set of $F\in \MA^3_3(\F_2)$ having the following criteria:
\begin{itemize}
\item $F$ is a mock automorphism,
\item $F=I+H$, $H$ homogeneous of degree 3.
\end{itemize}
 We found
$1520$ endomorphisms satisfying the above requirements. The tabular lists them in 20 classes up to conjugation by  linear maps:\\
\ \\
\begin{tabular}{|l|l|l|l|c|}
 \hline
& Representant &  Bijection over & $\#$\\ \hline \hline
{\bf 1.}& $\mathbf{(x,y,z)}$ \\ \hline
1a. &$(x,y,z)$ &all & 1\\ \hline
1b. &$(x,y,z+x^2y+xy^2)$ &all & 7 \\ \hline
1c. &$(x,y,z+x^3+x^2y+y^3)$ &all & 14 \\ \hline
1d. &$(x,y+x^3,z+x^3)$ &all & 21 \\ \hline
1e. &$(x,y,z+x^3+x^2y+xy^2)$ & all & 21\\ \hline
1f. &$(x,y,z+x^2y)$ & all & 42\\ \hline
1g. &$(x,y+x^3,z+xy^2)$ &all & 42\\ \hline
1h. &$(x,y+x^3,z+x^2y+xy^2)$ & all & 42 \\ \hline
1i. &$(x,y + z^3,z+x^2y)$ &all & 42 \\ \hline
1j. &$(x,y+x^3,z+ x^2y + y^3)$ & all & 84\\ \hline
1k. &$(x,y+x^3,z+y^3)$ & all & 84\\ \hline
{\bf 2.} & $\mathbf{(x,y, z+x^3z^4+xz^2)}$\\\hline
 2&$(x,y+x^3+xz^2,z+xy^2+xz^2)$ & $\F_2,\F_4,\F_{16},\F_{32}$ & 56\\ \hline
{\bf 3.}&$\mathbf{ (x,y,z+x^3z^2+x^3z^4)}$ \\ \hline
3a. &$(x,y+ xz^2,z+x^2y + xy^2)$ & $\F_2,\F_4$ & 84 \\ \hline
3b. &$(x,y+xz^2,z+x^3+x^2y+xy^2)$ & $\F_2,\F_4$ & 84\\ \hline
{\bf 4.}& $\mathbf{(x,y,z+xz^2+xz^6)}$ \\ \hline
4a. &$(x,y+x^3+z^3,z+x^3+xy^2+xz^2)$ & $\F_2$ & 168\\ \hline
4b. &$(x,y+z^3,z+xy^2+xz^2)$ & $\F_2$& 168\\ \hline
{\bf 5.}&$\mathbf{ (x,y,z+x^3z^2+xy^2z^4+x^2yz^4+x^3z^6)}$ \\ \hline
5a. &$(x,y+xz^2,z+xy^2+y^3)$ & $\F_2$ & 168\\ \hline
5b. &$(x,y+xz^2,z+x^3+x^2y+y^3)$ & $\F_2$ & 168\\ \hline
{\bf 6.}& $\mathbf{(x,y,z+x^3z^2+xy^2z^2+x^2yz^4+x^3z^6)}$ \\ \hline
6. &$(x,y+xy^2+xz^2,z+x^3+x^2y)$ &  $\F_2$ & 168\\ \hline
{\bf 7.}& $\mathbf{(x+y^2z,y+x^2z+y^2z,z+x^3+xy^2+y^3)}$ \\ \hline
7. &$(x+y^2z,y+x^2z+y^2z,z+x^3+xy^2+y^3)$ & $\F_2$ & 56\\ \hline
\end{tabular}\\
\ \\
The first column gives a representant up to linear conjugation, and the bold fonted one gives a representant under tamely equivalence for the classes listed beneath it. 
The second column lists for which field extensions (from $\F_{2^r}$ where $1\leq r\leq 5$) the map is also a bijection of $\F_{2^r}^3$
Class 1 are the 400 automorphisms, all of them are tame and sastisfy the dependence conjecture. All classes 
are tamely equivalent to a map of the form $(x,y,P(x,y,z))$, except the last class 7 - these maps do not satisfy the Dependence Criterium, which makes them very interesting! 

The above tabular might make one think that any mock automorphism in $\MA_3(\F_2)$ of the form $F= (x,y+H_2, z+H_3)$ where $H_2,H_3$ are homgeneous of the same degree, then one can tamely change the map into one of the form $(x,y,z+K)$, but the below conjecture might give a counterexample:

\begin{conjecture} Let $F=(x, y+y^8z^2+y^2z^8, z+y^6z^4+y^4z^6)$. Then $F$ is not tamely equivalent to 
a map of the form $(x,y,z+K)$ .
\end{conjecture}

Due to our lack of knowledge of the automorphism group $\TA_3(\F_2)$, this conjecture is a hard one unless one finds a good invariant
of maps of the form $(x,y, z+K)$. 

\subsection{The finite field of three elements: $\mathbb{F}_3$}\label{F3}

\subsubsection{Degree 2 over $\F_3$}

Over $\F_3$, there are $(27-1)(27-3)(27-9)=11232$ linear automorphisms and $27*11232=303264$ affine automorphisms.\\
From corollary \ref{Jac1aut} it follows that if $\det(\Jac(F))=1$ and $\deg(F)\leq 2$,  then $F$ is an automorphism - so we will not encounter any mock automorphisms which aren't an automorphism in this class.
There are $2835$ automorphisms of degree less or equal to $2$ having affine part identity, so there are $2835\cdot 303264=$
automorphisms of degree $2$ or less. They all turned out to be tame.

\subsubsection{Locally finite}\label{locFin32}

We computed all  conjugacy classes under linear maps of locally finite automorphisms of $\F_3^3$ (see lemma
 \ref{CompLocFin}). There are $80$ orbits of affine automorphisms, composing a representative of each class with all of the $2,835$ tame automorphisms, gives us $226,800$ representatives of ``conjugacy classes''. We checked for each of them whether it was locally finite or not.
It turns out that $25,872$ of these conjugacy classes are locally finite. And there are exactly a hundred different minimum polynomials that can appear. 

Of the appearing minimal polynomials in this list, all polynomials of degree 3 appear in this list. The highest minimum polynomials are of degree 10. We list just a (sort of random, non-affine) ten minimum polynomials, their order (which is determined by the minimum polynomial), number of {\em conjugacy classes} with this minimum polynomial, and one example.  The reader interested in the complete list we refer to chapter 6 of the Ph.-D. thesis of the second author \cite{RoelThesis}.\\
\ \\
{\tiny
\begin{tabular}{|l|c|c|l|}
\hline
Minimum polynomial & order & $\sharp$ & example\\ \hline \hline
$F^2+2I$ & 2 & 509 & $\left(\begin{array}{l} 2x^2+xy+xz+2x+y^2+z^2\\ 2x^2+xy+xz+y^2+2y+z^2\\ 2x^2+xy+xz+y^2+z^2+2z\end{array}\right)$ \\ \hline
$F^3+F^2+2F+2I$ & 6 & 5084 & $\left(\begin{array}{l} x^2+xy+2x+y^2\\ x^2+xy+y^2+2y\\ 2x^2+2y^2+2z\end{array}\right)$ \\ \hline
$F^4+2F^2+2F+2I$ & 24 & 2 & $\left(\begin{array}{l} x^2+xz+2x+y^2+2y+z^2\\ 2x+y+z\\ x^2+xz+x+y^2+y+z^2+z\end{array}\right)$ \\ \hline
$F^4+2F^3+2F+I$ & 9 & 3804 & $\left(\begin{array}{l} 2x^2+2xy+x+2y^2+1\\ 2x^2+2xy+2y^2+y+1\\ 2x^2+xy+2x+2y^2+z+1\end{array}\right)$\\ \hline
$F^4+F^3+F^2+2F+I$ & 8 & 38 & $\left(\begin{array}{l} x^2+2xy+xz+x+y^2+yz+z^2+2z+2\\ x^2+2xy+xz+y^2+yz+z^2+2z\\ 2x^2+xy+2xz+2x+2y^2+2yz+2y+2z^2+2\end{array}\right)$ \\ \hline
$F^5+2F^3+2F^2+F+2I$ & 8 & 8 & $\left(\begin{array}{l} 2x^2+xy+xz+y^2+2y+z^2\\ 2x^2+xy+xz+2x+y^2+2y+z^2+z\\ 2x^2+xy+xz+x+y^2+z^2+z\end{array}\right)$ \\ \hline
$F^6+F^5+2F^4+F^3+2I$ & 24 & 16 & $\left(\begin{array}{l} y^2+yz+2y+z^2\\ 2x+y^2+yz+2y+z^2+z\\ x+y^2+yz+z^2+z\end{array}\right)$ \\ \hline
$F^7+F^6+2F+2I $ & 18 & 396 & $\left(\begin{array}{l} 2x^2+2xz+2y^2+2y+2z^2+2z+1\\ x^2+xz+2y+z^2\\ 2x^2+2xz+2x+2y^2+2y+2z^2+2\end{array}\right)$ \\ \hline
$F^{10}+F^8+2F^5+F^2+2F+2I $ & 26 & 40 & $\left(\begin{array}{l} y+2z^2+z+1\\ x^2+2xz+x+z^2+1\\ x+z+1\end{array}\right)$ \\ \hline
$F^{10}+F^9+2F^8+F^7+F^6+F^5+2F^3+2F+I$ & 13 & 48 & $\left(\begin{array}{l} 2x^2+2xy+2xz+y^2+yz+y+2z^2+2z+1\\ 2x+y+z+2\\ 2x^2+2xy+2xz+2x+yz+y+2z^2+2z+1\end{array}\right)$ \\ \hline
\end{tabular}}

\subsubsection{Degree 3 over $\F_3$}

The amount of elements in $\overline{\ME}_3(\F_3)$ of the form $(x,y,z)+(0,H_2,H_3)$ (i.e. satisfying the dependency criterion)
where $H_2,H_3$ are homogeneous of degree 3 is too large: this set has $3^{20}$ elements which was too large for our system to scan through; however, we think that this case is feasible for someone having a stronger, dedicated system and a little more time.

\subsection{The finite fields $\mathbb{F}_4$ and  $\mathbb{F}_5$}\label{F4F5F7}

In this section we will only restrict to degree 2, and to the maps which satisfy the dependency conjecture.
Thus, in this section we restrict to maps  $F$ of the form $(x+H_1,y+H_2,z)$ where $H_1,H_2$ are of degree 2. 

\subsection{The finite field $\F_4$}

There are $(64-1)(64-4)(64-16)=181,440$ linear automorphisms and $64*181,440=11,612,160$ affine automorphisms.
We considered the follwing maps:
\begin{itemize}
\item $F\in \overline{\ME}_3^2(\F_4)$,
\item $F$ is a mock automorphism,
\item $F$ is of the form $(x+H_1(x,y,z), y+H_2(x,y,z),z)$ (i.e. $F$  satisfies the  dependency criterion).
\end{itemize}
and we counted $40,384$ such maps.
Under tame equivalence, we have the following classes:
\begin{itemize}
\item[1] $(x,y,z)$ (tame automorphisms)
\item[2] $(x+x^2+x^4, y, z)$ 
\end{itemize}
So, surprisingly, we only find a subset of the classes we found over $\F_2$. Well, not really surprising - the dependency criterion removes the classes 3 and 4 of theorem \ref{T1} from the list. We conjecture that the four classes of theorem \ref{T1} are the same for $\F_4$:

\begin{conjecture} (i) Suppose $F\in \ME_3^2(\F_4)$ is a mock automorphism of $\F_4$.
Then $F$ is tamely equivalent to $(P(x),y,z)$ where \[ P=x, P=x^4+x^2+x, P=x^8+x^4+x, \textup{~or~}P=x^8+x^2+x.\]
(ii) Suppose $F\in \ME_3^2(L)$ is a mock automorphism of $L$, where $[L:\F_2]<\infty$.
Then $F$ is tamely equivalent to $(P(x),y,z)$ where \[ P=x, P=x^4+x^2+x, P=x^8+x^4+x, \textup{~or~}P=x^8+x^2+x.\]
If $3|[L:\F_2]$ then one should remove the class of $P=x^4+x^2+x$, and if $7|[L:\F_2]$ then one should remove the classes of 
$ P=x^8+x^4+x$ and $P=x^8+x^2+x$. 
\end{conjecture}

It would be interesting to see a proof of this conjecture by theoretical means - or a counterexample of course.

\subsection{The finite field $\F_5$}

There are $(125-1)(125-5)(125-25)=1,488,000$ linear automorphisms and $125\cdot1,488,000=1,186,000,000$ affine automorphisms.
We consider maps of the following form:
There are $3,625$ mock automorphisms of $\F_5$ of degree at most 2. endomorphisms satisfying the following:
\begin{itemize}
\item $F\in \overline{\ME}_3^2(\F_5)$,
\item $F$ is a mock automorphism of $\F_5$, 
\item $F$ satisfies the  dependency criterion (i.e. $F=(x+H_1(x,y,z), y+H_2(x,y,z),z)$).
\end{itemize}
We counted $3,625$ such maps - and becaus of  Corollary \ref{Jac1aut}, they are all automorphisms.
They all turned out to be tame maps.

\section{Conclusions}

We can gather some of the results in the below theorem:

\begin{theorem} Let $F\in \GA_3^d(\F_q)$. If one of the below conditions is met, then $F$ is tame:
\begin{itemize}
\item $d=3$, $q=2$,
\item $d=2$, $q=3$,
\item $d=2$, $q=4$ or $5$, and $F$ satisfies the Dependency criterion.
\end{itemize}
\end{theorem}

This gives rise to the following conjecture:

\begin{conjecture} If $F=I+H\in \GA_n(k)$ where $H$ is homogeneous of degree 2, then $F$ is tame.
\end{conjecture}

This natural conjecture might have been posed before, but we are unaware. This article proves this conjecture for $n=3$ and $k=\F_2,\F_3,\F_4,\F_5$. We expect that for $n=3$ and a generic field a solution is within reach.\\

Unfortunately, the computations did not allow us to go as far as finding some candidate non-tame automorphisms (though the Nagata automorphism is one, however it is of too high degree). However,
one of the interesting conclusions is that the set of {\em classes} (under tame automorphisms) of mock automorphisms  seems to be much smaller than we originally expected: only 4 (perhaps 3) over $\F_2$ up to degree 2, and at most 7 over $\F_2$ of degree 3.
In particular, we are puzzled by the interesting question  
whether the two endomorphisms over $\F_2$  described by $(x^8+x^4+x,y)$ and $(x^8+x^2+x,y)$ are not equivalent, as stated in question \ref{Qu1}.

{\bf Computations:} For computations we used the MAGMA computer algebra program. The reader interested in the routines we refer to chapter 6 of the thesis of the second author, \cite{RoelThesis}. Also, we posess databases usable in MAGMA, which we hope to share in the near future on a website. \\

{\bf Acknowledgements:} The second author would like to thank Joost Berson for some useful discussions.

\end{document}